\documentclass[12pt]{article}

\usepackage{amsmath,amssymb,amsthm,graphicx}
\usepackage{a4wide}
\newcommand{\Prob}{\mathbb{P}}
\newcommand{\mumin}{\mu_{\textnormal{min}}^c}
\newcommand{\Emu}{E_c[\mu]}
\newcommand{\dEmu}{\frac{\delta E_c}{\delta \mu}}
\newcommand{\rhoz}{\rho^0}

\newcommand{\dmut}{d\mu_{\textnormal{min}}^c(x)/dx}
\newcommand{\dmu}{\frac{d\mu_{\textnormal{min}}^c}{dx}(x)}

\newcommand{\Lc}{ \overline{L}_c^{\Gamma}(z) }
\newcommand{\LcI}{ \overline{L}_c^{I} }
\newcommand{\dd}{{\rm d}}
\newcommand{\wq}{\widehat{w}_{2N}}
\newcommand{\Kh}{\widehat{K}_{2N,2k}}
\newcommand{\Rh}{\widehat{R}^{(2N,2k)}}
\newtheorem{theorem}{Theorem}
\newtheorem{lemma}{Lemma}
\newtheorem{prop}[lemma]{Proposition}

\newtheorem{fact}[lemma]{Fact}
\newtheorem{definition}[lemma]{Definition}

\numberwithin{lemma}{section}
\numberwithin{theorem}{section}
\numberwithin{equation}{section}

\title{An ensemble related to discrete orthogonal polynomials and its application to tilings of a half-hexagon.}

\author{Uwe~Schwerdtfeger\\
\\
Fakult\"at f\"ur Mathematik\\
Technische Universit\"at Chemnitz\\
09107 Chemnitz, Germany\\
}

\begin{document}
\maketitle

\begin{abstract}
We study a family of discrete probability measures that may be used to model particle configurations on a set of discrete nodes in presence of an impenetrable wall. The correlation functions are shown to be determinantal and can be expressed in terms of discrete orthogonal polynomials. We make strong use of the results of the monograph \cite{BKMM07} to explain the asymptotic behaviour of these ensembles. Our general results are applied to random tilings of the half-hexagon, a model introduced in \cite{ForNor09}. In regions away from the bottom boundary the model exhibits the same asymptotic behaviour as tilings of the full hexagon, including an ``arctic half-ellipse.'' Close to that boundary, however, the statistics are different.

\end{abstract}

\section{Introduction}
In this article we are concerned with discrete probability distributions on particle configurations on a discrete set of nodes. Given a positive weight function $w_N$ defined on a set of nodes $Y_N=\{y_{N,0}<\ldots < y_{N,N-1}\}$ contained in an interval $(0,b],$ we define 
$DOPE^{\textnormal{sym}}(N,k)=DOPE^{\textnormal{sym}}(N,k,w_N)$ to be the ensemble of $k$-tuples
$(x_1,\ldots,x_k)\in Y_N^k,$ $x_1<\ldots<x_k,$ equipped with the probability distribution 
\begin{equation}\label{dopesym}
p_{\textnormal{sym}}^{(N,k)}(x_1,x_2,\ldots,x_{k})=\frac{1}{Z_{N,k}^{\textnormal{sym}}}\prod_{i=1}^{k}x_i^2w_{N}(x_i)\prod_{1\leq i<j\leq k}\left(x_j^2-x_i^2\right)^2.
\end{equation}
$Z_{N,k}^{\textnormal{sym}}$ is a normalisation constant such that the sum over all configurations equals one. The name reflects the close relation to the following well-studied \emph{discrete orthogonal polynomial ensemble} $DOPE(2N,2k)=DOPE(2N,2k,\wq).$ Consider the set of nodes $X_{2N}=-Y_{N}\cup Y_N$ and the symmetric extension $\wq$ of $w_N$ 
\begin{equation}\label{def:X2Nwhat}
 X_{2N}=-Y_{N}\cup Y_N,\quad \wq(x)=\begin{cases}w_N(x),\;x\in Y_N,\\w_N(-x)\;x\in -Y_N.\end{cases}
\end{equation}
Then $DOPE(2N,2k)$ is the probability distribution on the set of $2k$-tuples $(x_1,\ldots,x_{2k})\in X_{2N}^{2k}$ with $x_1<x_2<\ldots<x_{2k}$ ($k\leq N$), given by
\begin{equation}\label{dope}
p^{(2N,2k)}(x_1,x_2,\ldots,x_{2k})=\frac{1}{Z_{2N,2k}}\prod_{i=1}^{2k}\wq(x_i)\prod_{1\leq i<j\leq 2k}(x_j-x_i)^2,
\end{equation}
and $Z_{2N,2k}$ is again a normalisation constant. Both probability measures $p^{(2N,2k)}$ and $p_{\textnormal{sym}}^{(N,k)}$ model repulsive particles since the Vandermonde type products penalise close and coincident particles. The additional feature of the measure $p_{\textnormal{sym}}^{(N,k)}$ is a repulsive wall at $0,$ modelled by the factors $x_i^2$ penalising particles near $0.$ The ensemble \ref{dope} has been studied in depth under the assumptions in Section \ref{analyticassumptions} in the monograph \cite{BKMM07}. The purpose of this article is to apply and extend these results to the ensemble \ref{dopesym}.

By $R^{(N,k)}_m$ and $\widehat{R}_m^{(2N,2k)}$ we denote the $m$-point correlation functions for $DOPE^{\textnormal{sym}}(N,k)$ and $DOPE(2N,2k),$ respectively, i.e.
\begin{equation}\label{def:mpoint}
R^{(N,k)}_m(x_1,\ldots,x_m)=\Prob(\text{particles at each of the sites }x_1,\ldots,x_m),
\end{equation}
analogously for $\widehat{R}^{(2N,2k)}_m.$ It is well known that $\Rh_m$ is determinantal, 
\begin{equation}
\Rh_m(x_1,\ldots,x_m)=\det\left( \Kh\left(x_i,x_j\right)  \right)_{i,j=1,\ldots,m},
\end{equation}
the kernel $\Kh$ involving the eponymous orthogonal polynomials.
$R^{(N,k)}_m$ turns out to be determinantal as well with a kernel 
\[
K_{N,k}(x,y)=\widehat{K}_{2N,2k}(x,y)-\widehat{K}_{2N,2k}(x,-y),\;x,y\in Y_N.
\]
A very short summary of our results is: under the assumptions of Section \ref{analyticassumptions} the contribution of $-\widehat{K}_{2N,2k}(x,-y)$ is negligible unless one considers nodes $x,y$ at $O(1/N)$ distance from $0.$ Hence, for positive nodes bounded away from 0,  $\widehat{R}^{(2N,2k)}_m$ and $R^{(N,k)}_m$ are asymptotically indistinguishable, the approximating correlation kernels are the sine and Airy kernels (cf. eqs. \eqref{def:sinekernel} and \eqref{def:Airykernel}). For a suitable choice of nodes at $O(1/N)$ distance to 0, $R^{(N,k)}_m$ exhibits a different behaviour governed by a modified sine kernel
\begin{equation*}
S^0(\xi,\eta):=\frac{\sin(\pi(\xi-\eta))}{\pi(\xi-\eta)}-\frac{\sin(\pi(\xi+\eta))}{\pi(\xi+\eta)}.
\end{equation*}
The former kernels are well-known from the theory of random matrices, they describe the bulk and edge correlations of the spectrum of a large GUE (Gaussian Unitary Ensemble) matrix, respectively \cite{Forrester08,Mehta04}. The latter kernel occurs in a study of non-intersecting Brownian excursions \cite{TraWid07} where it describes the joint distributions of the excursions at a given point in time. Our main application may be viewed as the discrete time analogue of that model.

\medskip

\noindent
\textbf{Outline of the paper:} In Section \ref{sec:halfhex} we introduce our main application, namely random tilings of the half-hexagon. In Section \ref{correlationfunctions} we show that $DOPE^{\textnormal{sym}}(N,k)$ has determinantal correlation functions and its relation to $DOPE(2N,2k)$. In Section \ref{analyticassumptions} we give all the necessary assumptions and definitions to state our general results in the subsequent Section \ref{statements}. In Section \ref{proofs} we prove the additional estimates required to apply the results from \cite{BKMM07}. In Section \ref{halfhexproofs} we prove the results from Section \ref{sec:halfhex}. 

\section{Random tilings of the half-hexagon}\label{sec:halfhex}

In this section we outline the the combinatorial model that inspired our studies. The $(Q,R,S)$\emph{-hexagon} is a hexagon with integer side lengths $Q,$ $R,$ $S,$ $Q,$ $R,$ $S$ and every angle equal to $120^\circ.$ We study tilings thereof with $60^\circ$ unit rhombi referred to as \emph{lozenges} or simply \emph{tiles}. The hexagon is filled without gaps and overlap, and no tile juts out beyond the boundary (``fixed boundary conditions"). The tiles occur in three different species (orientations) referred to as \emph{up-, vertical} and \emph{down-}tiles, see figure \ref{themodels}.  As we are interested in a certain symmetry class we restrict to hexagons with $Q=2k$ and $R=S.$
To quantify things, we fix an ON-coordinate system and look at (symmetric) tilings of the $(2k,R,R)$-hexagon with corners $(\pm \sqrt{3}R/2,\pm k)$ and $(0,\pm(k+R/2).$ The corners are numbered in counterclockwise order starting with $P_1=(- \sqrt{3}R/2,-k).$ If we focus upon the symmetry class of tilings w.r.t. the reflection in the $x$-axis, we can throw away the vertical tiles on the $x$-axis enforced by the symmetry constraint and the part below the $x$-axis to obtain a tiling of the so-called $(k,R)$\emph{-half-hexagon}, a model studied in \cite{ForNor09}.

Recall that tilings of the full $(2k,R,R)$-hexagon map bijectively to families of $2k$ non-intersecting paths on the triangular point lattice 
\[
L=\left\{\left(\frac{-\sqrt{3}}{2}R,\frac{1}{2}\right)+q\left(\frac{\sqrt{3}}{2},\frac{1}{2}\right)+r\left(\frac{\sqrt{3}}{2},-\frac{1}{2}\right),\; q,r\in\mathbb{Z} \right\}
\]
with starting points in the set $S=S^+\cup S^-$ and end points in $E=E^+ \cup E^-,$ where 
\[
S^+=\left\{ \left(-\frac{\sqrt{3}}{2}R,i+\frac{1}{2}\right),\;i=0\ldots k-1 \right\},\;E^+=\left\{ \left(\frac{\sqrt{3}}{2}R,i+\frac{1}{2}\right),\;i=0\ldots k-1 \right\},
\]
and $S^-$ (resp. $E^-$) denotes the reflection of $S^+$ (resp. $E^+$) in the $x$-axis. Admissible steps are $(x,y)\rightarrow(x+\sqrt{3}/2,y\pm1/2).$ To see this, connect in each up- and in each down-tile the midpoints of the vertical sides by a straight line segment (the decoration depicted in figure \ref{themodels}). 

\begin{figure}[t]
\begin{center}
\includegraphics[height=76mm,width=147mm]{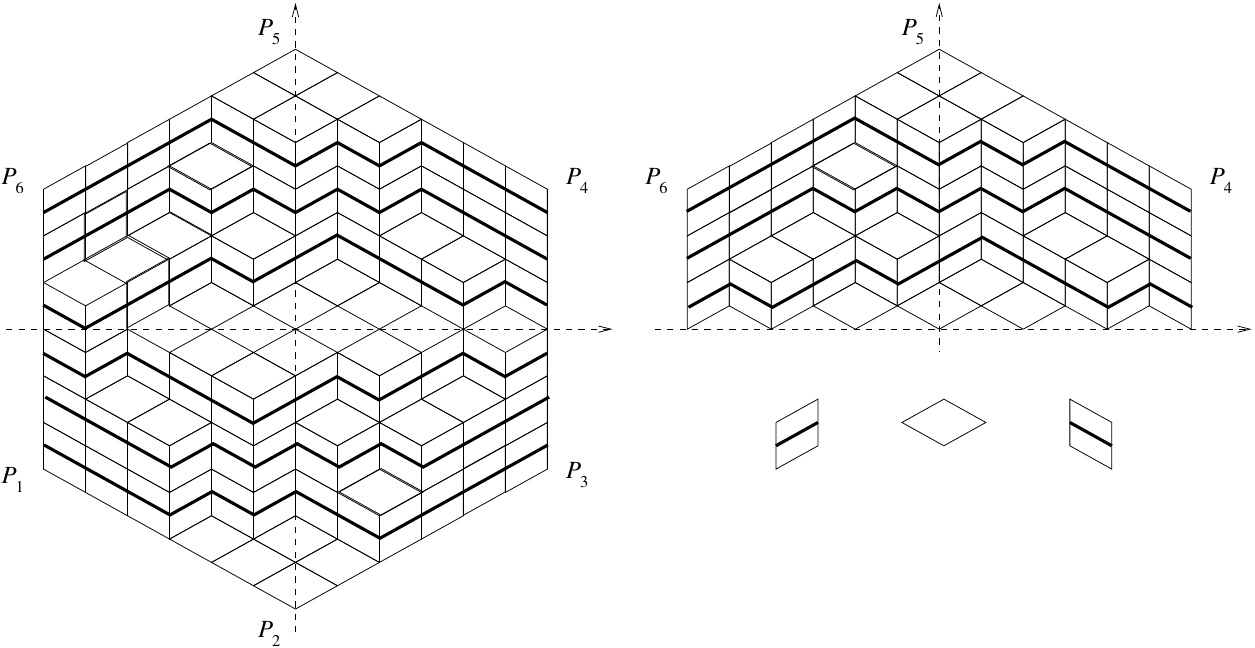}
\caption{\small{Tiling of a $(2k,R,R)$-hexagon (without symmetry), tiling of a $(k,R)$-half-hexagon, up-, vertical  and down-tiles}
}\label{themodels}
\end{center}
\end{figure}
In the same fashion tilings of the half-hexagon are mapped to families of $k$ non-intersecting paths with starting points in $S^+$ and end points in $E^+$ which \emph{do not touch the $x$-axis}. Denote by $L_m,\;m=0,\ldots,2R$ the set of numbers such that
$\{(-R+m)\sqrt{3}/2\}\times L_m$ is the intersection of the vertical line $x=(-R+m)\sqrt{3}/2$ with the lattice $L$ and the $(2k,R,R)$-hexagon. Note that by symmetry $L_m=L_{2R-m}$ and, for $m\leq R,$ $\left |L_m \right |=2k+m.$ Then $\{(-R+m)\sqrt{3}/2\}\times L_m$ is the set of possible points where a family of lattice paths corresponding to a tiling can intersect a vertical line after $m$ steps. Denote by $L_m^+$ the set of positive elements of $L_m.$ Now 
\begin{prop}[Theorem 4.1 in \cite{Johansson02}, Lemma 2.2 in \cite{ForNor09}]\label{hexagondistributions} Consider the sets of families of $2k$ non-intersecting lattice paths with starting points in $S$ and end points in $E$ (tilings of the $(2k,R,R)$-hexagon) and of families of $k$ such lattice paths with starting points in $S^+$ and end points in  $E^+$ not touching the $x$-axis ($(k,R)$-half-hexagon) to be equipped with their respective uniform distributions.
\begin{enumerate}
\item Let $x_1<x_2<\ldots<x_{2k},$ $x_i\in L_m$ for $i=1,\ldots, 2k.$ Then the probability of a family of non-intersecting lattice paths to intersect the vertical line $x=(-R+m)\sqrt{3}/2$ at ordinates $x_1,\ldots,x_{2k}$ is equal to
\begin{equation}\label{fullhexagon}
P_{m}(x_1,x_2,\ldots,x_{2k})=\frac{1}{Z_{m}}\prod_{i=1}^{2k}w(x_i)\prod_{1\leq i<j\leq 2k}(x_j-x_i)^2.
\end{equation}
\item Let $x_1<x_2<\ldots<x_{k},$ $x_i \in L_m^+$ for $i=1,\ldots, k.$ Then the probability of a family of $k$ non-intersecting lattice paths not touching the $x$-axis to intersect the vertical line $x=(-R+m)\sqrt{3}/2$ at ordinates $x_1,\ldots,x_{k}$ is equal to
\begin{equation}\label{halfhexagon}
P_{m}^\textnormal{sym}(x_1,x_2,\ldots,x_{k})=\frac{1}{Z_{m}^\textnormal{sym}}\prod_{i=1}^{k}x_i^2w(x_i)\prod_{1\leq i<j\leq k}\left(x_j^2-x_i^2\right)^2.
\end{equation}
\end{enumerate}
In both cases the weight function $w$ is \emph{even} and it is equal to
\begin{equation}\label{weighthexagon}
\begin{split}
w(z)&=\frac{1}{ (m/2+k-1/2 - z)!(R-m/2+k-1/2 - z)! }\\
& \times \frac{1}{ (m/2+k-1/2 + z)!(R-m/2+k-1/2 + z)! },\\
\end{split}
\end{equation}
$Z_{m}$ and $Z_{m}^\textnormal{sym}$ are normalisation constants.
\end{prop}
\noindent
The ensembles \eqref{halfhexagon} and \eqref{fullhexagon} are of $DOPE^{\textnormal{sym}}(N,k)$ and  $DOPE(2N,2k)$ types (with a suitable choice of $N$), respectively, and since $\tilde{w}(z)$ is symmetric they are related in precisely the fashion outlined above. Note that every point in $\{(-R+m)\sqrt{3}/2\}\times L_m$ where none of the paths intersect is a centre point of a vertical tile. So the ``particle-hole-dual" ensemble of \eqref{fullhexagon} (which is also of $DOPE$ type) describes the distribution of vertical tiles along the line $x= (-R+m)\sqrt{3}/2.$ In \cite{ForNor09} the joint distribution of vertical tiles on \emph{all} lines $L_1,\ldots,L_m$ for some $m<R$ is considered.  

We consider asymptotics of hexagon tilings when the side lengths grow proportionally, $R/k\longrightarrow\lambda>0$ or, after a rescaling by $1/k,$ tilings of a half hexagon of sidelenghts $1$ and $\lambda$ with ever finer tiles. The corners of the limiting half-hexagon are at $(\pm \sqrt{3}\lambda/2,0)$, $(\pm \sqrt{3}\lambda/2,1)$ and $(0,1+\lambda/2).$ Consider the vertical line $L_\tau$ with ordinate $\tau\in (\sqrt{3}\lambda/2,0]$ intersecting the rescaled $(k,R)$-half-hexagon. 

\smallskip
\noindent
\textbf{Remark:} In the following we make the obvious abuse of notation: when speaking of nodes etc. ``on $L_\tau$" we mean of course a line of the rescaled lattice approximating $L_\tau.$

\smallskip
\noindent
The abscissas of the nodes (possible intersection points) lie in the interval $[0,c(\tau)^{-1}]$, where $c(\tau)^{-1}$  is the length of the portion of $L_\tau$ inside the half-hexagon. Hence $c=c(\tau)$ is the asymptotic ratio $k/N_\tau$ of intersection points and nodes on $L_\tau.$ 
Denote by $\beta=\beta(\tau)>0$ the abscissa of intersection point of $L_\tau$ with the inscribed half-ellipse and by $(\tau_0,\beta(\tau_0))$ the tangent point of the half-hexagon and its inscribed half-ellipse, i.e.
\begin{equation}
c=\frac{1}{\tau/\sqrt{3}+1+\lambda/2},\quad\beta=\sqrt{1+\lambda}\sqrt{1-\left(\frac{2\tau}{\lambda \sqrt{3}}\right)^2 },\quad\tau_0=-\frac{\sqrt{3}\lambda^2}{2(2+\lambda)}.
\end{equation}
With some further notation
\[
T(x)=\sqrt{\frac{\beta-x}{\beta+x} } ,\quad k_1=\sqrt{  \frac{1-2c\,\tau/\sqrt{3}   +c\beta}{1-2c\,\tau/\sqrt{3}    -c\beta}   },\quad k_2=\sqrt{  \frac{1 +c\beta}{1 -c\beta}   }
\]
we define a function $H(\tau,x)$ on the interval $[0,c(\tau)^{-1}]$  for every $\tau\in (\sqrt{3}\lambda/2,0]\setminus\{\tau_0\}$ by
\begin{equation*}
H(\tau,x)=\begin{cases}
\begin{aligned}
\frac{1}{2}-\frac{ 1 }{ 2\pi } &\left(\arctan(k_1^{-1}T(x))-\arctan(k_1T(x))\right.\\ &+ \left. \arctan(k_2T(x))-\arctan(k_2^{-1}T(x))\right.\\
 \end{aligned}\; &\text{if }\tau<\tau_0,\; x\in [0,\beta),\\
1/2\; &\text{if }\tau<\tau_0,\; x \in [ \beta,c^{-1}],\\
\begin{aligned}
\frac{ 1 }{ 2\pi } &\left(  \arctan(k_1T(x))-\arctan(k_1^{-1}T(x))\right.\\
 &\left.+\arctan(k_2T(x))-\arctan(k_2^{-1}T(x))\right) 
\end{aligned} &\text{if } \tau>\tau_0,x\in [0,\beta),\\
0\;&\text{if }\tau>\tau_0,\; x \in [ \beta,c^{-1}].\\
\end{cases}
\end{equation*}
$H(\tau,\cdot)$ is the density function of the associated equilibrium measure, cf. Section \ref{equilibrium}. The following theorem shows that away from the bottom boundary, the particle statistics of the half-hexagon and full hexagon ensembles are asymptotically indistinguishable. 
\begin{theorem}\label{theo:hhbulkandrightedge}
Denote by $R^{(N_\tau,k)}_m$ the $m$-point correlation function of intersection points on $L_\tau.$ Then we have, with $\dmut=H(\tau,x)$ and $\rhoz(x)\equiv c/2:$
\begin{enumerate}
\item For a fixed $x$ in the interval $[0,\beta)$ the estimates \eqref{bandest} and \eqref{bandzero} of Theorem \ref{correlation:band} hold accordingly.
\item For $\tau>\tau_0$ the ``void" cases of Theorems \ref{correlation:edge} and \ref{Theorem39} hold, and for $\tau<\tau_0$ the respective ``saturated region" cases hold. 
\item For any closed interval $J\subset (\beta, c^{-1}]$ the first estimate of Theorem \ref{correlation:gap} holds for $\tau>\tau_0$ and the second estimate holds for $\tau<\tau_0.$
\end{enumerate}
In particular, if $\xi_k$ is a sequence of nodes with $\xi_k\longrightarrow x\in(0,c^{-1}]$ the one-point correlation function $R^{(N_\tau,k)}_1(\xi_k)$ converges pointwise to $2H(\tau,x).$ The rate of convergence is $O(1/k)$ for $x<\beta$ and uniformly exponentially fast in any closed subset $J\subset (\beta,c^{-1}].$ 
\end{theorem}

\noindent
\textbf{Remark:} The assertion on the one-point correlation function yields that, for $\tau<\tau_0,$ above the inscribed ellipse the line $L_\tau$ is fully packed with intersection points (and hence with the blue ``up-tiles"), uniformly with probability exponentially close to one. If $\tau_0<\tau\leq 0,$ the part of $L_\tau$ above the ellipse is devoid of intersection points and hence fully packed with vertical (red) tiles with probability exponentially close to one. Close to a point with ordinate $>0$ in the interior of the ellipse, the one-point correlation function takes a value strictly between 0 and 1 and hence both intersection points and vertical tiles occur with positive probability.  This reflects the ``arctic phenomenon" present in the large tiling of the half-hexagon depicted in Figure \ref{arcticsemicircle}, namely the tiling being highly ordered at the corners and unordered in the interior, with a sharp transition along the inscribed ellipse.
\begin{figure}[t]
\begin{center}
\includegraphics[height=90 mm,width= 150 mm]{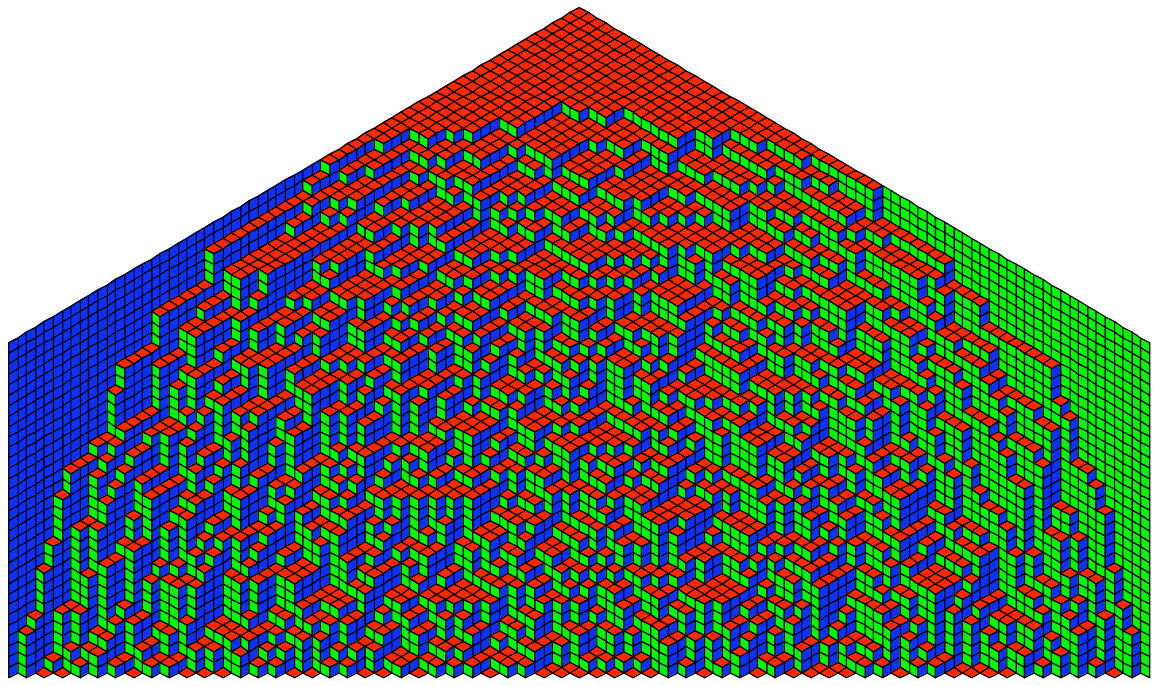}
\caption{\small Arctic phenomenon in a $(32,64)$-half-hexagon}\label{arcticsemicircle}
\end{center}
\end{figure}

\medskip

The novel feature of the half-hexagon tilings is the behaviour at the bottom boundary which is asymptotically governed by the kernel
\[
 {\cal S}^{0}_{ij}=\frac{\sin\left(2H(\tau,0)  \pi(i-j) \right)}{\pi(i-j)} -\frac{\sin\left(2H(\tau,0) \pi(i+j) \right)}{\pi(i+j)}.
\]

\begin{theorem}\label{leftmostparticlehex}
Let the nodes on $L_\tau$ be the set $Y_{N_\tau}=\{y_{N_\tau,j}=(2j+1)/2k,\;j=0,\ldots,N_\tau  \},$ i.e. $L_\tau$ is met by the lattice path family after an even number of steps.  For a fixed set $\mathbb{B}=\{j_1<j_2<\ldots<j_l\}$ of non-negative integers consider the subset $B=\{y_{N,j_1},\ldots,y_{N,j_l} \}\subset Y_{N_\tau}.$ Then
\begin{equation*}
\begin{split}
\Prob(\textnormal{exactly }m\textnormal{ intersection points in }B)=\frac{1}{m!}\left.\left(-\frac{\dd}{\dd t}\right)^m\det\left(\mathbb{I}-t\left.{\cal S}^0\right|_{\mathbb{B}} \right)\right|_{t=1}+O(1/k),
\end{split}
\end{equation*}
In particular, for the intersection height $x_\tau^{\textnormal{min}}$ of the bottom path with $L_\tau$ we have the limiting distribution function
\begin{equation}\label{xminhh}
\lim_{k\rightarrow\infty} \Prob\left( x_\tau^{\textnormal{min}} \geq \frac{s}{k} \right)
=\det\left(\mathbb{I}- \left.{\cal S}^{0}\right|_{\mathbb{B}(s)}\right),
\end{equation} 
where $\mathbb{B}(s)$ is the finite set of integers $\left\{0,1,\ldots, \left\lfloor s-1/2 \right\rfloor  \right\}$ and $\left\lfloor s-1/2 \right\rfloor$ denotes the integer part of $s-1/2.$ 
\end{theorem}

\noindent
\textbf{Remark:} Imagine we are given two samplers creating tilings of the $(2k,R,R)$-hexagon, one sampling uniformly from all possible tilings and the other only from those with a symmetry w.r.t. the horizontal axis  (viz. half-hexagon tilings). Unfortunately both samplers have been given ambiguous names and hence we wish to test the hypothesis ``horizontal symmetry" against the alternative ``unconstrained". To that end choose a large $k$ and $R$ and consider the statistic $x_{\textnormal{min}},$ which is the smallest positive height of an intersection point on the line $L_0.$ The limiting distribution function of $x_\textnormal{min}$ for the ensemble with the symmetry is given by the previous theorem, for the full ensemble it is given by the same formula with the operator  ${\cal S}^0$ replaced by ${\cal S}$ with the kernel 
\[
 {\cal S}_{ij}=\frac{\sin\left(2H(\tau,0)  \pi(i-j) \right)}{\pi(i-j)}. 
\]

The previous theorem in principle allows to statistically distinguish the ensembles of symmetric tilings of the full $(2k,R,R)$-hexagon from those without symmetry constraints.

\section{Correlation functions}\label{correlationfunctions}

In this section we compute the correlation functions of the ensemble $DOPE^{\text{sym}}(N,k)$ and derive the close relationship with 
$DOPE(2N,2k).$ We first summarise some facts on the $DOPE(2N,2k)$ ensembles \eqref{dope}, the most important of which is the determinantal structure of the correlation functions. It involves the eponymous orthogonal polynomials. A scalar product on the set of complex valued functions on the set of nodes $X_{2N}=-Y_N\cup Y_N$ is associated to the weight function $\wq$ via
\begin{equation}\label{scalarproduct}
(f,g)\mapsto \sum_{i=0}^{2N-1} \wq\left(x_{2N,i}\right)f\left(x_{2N,i}\right)
\overline{g\left(x_{N,i}\right)}.
\end{equation}
By applying the Gram-Schmidt procedure to the sequence of monomials $1,x,\ldots x^{2N-1}$ we obtain a family of \emph{orthonormal polynomials} $p_{2N,0},\ldots ,p_{2N,2N-1},$ i.e. the degree of $p_{2N,j}$ is equal to $j$ and the relation
\[
\sum_{i=0}^{2N-1} \wq\left(x_{N,i}\right)p_{2N,k}\left(x_{N,i}\right)
\overline{p_{2N,l}\left(x_{2N,i}\right)}=\delta_{kl}
\]
holds. Note that since the nodes, the weights and the coefficients of the $p_{2N,k}$ are real, we can omit complex conjugation. Furthermore the leading coefficient $\gamma_{2N,k}$ of $p_{N,k}$ is assumed to be positive. Denote by $\pi_{2N,k}:=\gamma_{2N,k}^{-1}p_{2N,k}$ the $k$th monic orthogonal polynomial. The $m$-point correlation functions $\Rh_m(x_1,\ldots,x_m)$ describe the probability that a configuration of $2k$ particles in $DOPE(2N,2k)$ contains particles at each of the $m$ sites $x_1,\ldots,x_m$ ($m\leq 2k$). We have
\begin{equation}
\begin{split}
\Rh_m(x_1,\ldots,x_m)&=\Prob(\text{particles at each of the sites }x_1,\ldots,x_m)\\
&=\det\left( \Kh\left(x_i,x_j\right)  \right)_{i,j=1,\ldots,m},
\end{split}
\end{equation}
where for $x,y \in X_{2N}$ the correlation kernel $\Kh(x,y)$ is given by
\begin{equation}\label{Christoffel}
\begin{split}
\Kh(x,y)&=\sqrt{\wq(x)\wq(y)}\sum_{n=0}^{2k-1}p_{2N,n}(x)p_{2N,n}(y)\\
&=\sqrt{\wq(x)\wq(y)}\cdot\frac{\gamma_{2N,2k-1}}{\gamma_{2N,2k}}\cdot
\frac{p_{2N,2k}(x)p_{2N,2k-1}(y)-p_{2N,2k}(y)p_{2N,2k-1}(x)}{x-y}
\end{split}
\end{equation}
if $x\neq y,$ and otherwise
\begin{equation}
\Kh(x,x)=\wq(x)\cdot\frac{\gamma_{2N,2k-1}}{\gamma_{2N,2k}}\cdot\left(
p_{2N,2k}^\prime(x)p_{2N,2k-1}(x)-p_{2N,2k-1}^\prime(x)p_{2N,2k}(x) \right).
\end{equation} 
The second ``$=$" in equation \eqref{Christoffel} is known as Christoffel-Darboux formula. The derivations of both the particular form of the correlation functions and the summation formula are fundamental calculations in random matrix theory, see \cite[Ch. 5]{Forrester08} or \cite[Ch. 5]{Mehta04}. 
\medskip

We can also obtain determinantal representations for the $m$-point correlation functions of the $DOPE^{\text{sym}}(N,k)$ ensembles \eqref{dopesym}, in particular the following Theorem shows how $DOPE^{\text{sym}}(N,k)$ and $DOPE(2N,2k)$ are related. This Theorem was pointed out to the author by P. J. Forrester (personal communication). A similar derivation involving Hermite polynomials is applied in \cite{TraWid07}, where correlation functions of $2n$ non-intersecting Brownian bridges and $n$ non-intersecting Brownian excursions are shown to be related in a similar fashion. 
\begin{theorem}
The $m$-point correlation function of the ensemble $DOPE^{\text{sym}}(N,k)$ ($m\leq k$) can be expressed as a determinant,
\begin{equation}
R_m^{(N,k)}(x_1,\ldots,x_m)=\det\left( K_{N,k}\left(x_i,x_j\right)  \right)_{i,j=1,\ldots,m}.
\end{equation}
The kernel $K_{N,k}$ can be expressed as
\[
K_{N,k}(x,y)=\widehat{K}_{2N,2k}(x,y)-\widehat{K}_{2N,2k}(x,-y),\;x,y\in Y_N,
\]
where $\Kh$ is the correlation kernel of $DOPE(2N,2k)$) given in eq. \eqref{Christoffel}.
\end{theorem}

\begin{proof}
In analogy to the construction of the $\pi_{2N,j}$ above by Gram-Schmidt orthonormalisation one
can construct monic polynomials $q_j(z)$ of degree $j,$ $j=0,\ldots,N-1$ with the following orthogonality property
\begin{equation*}
\sum_{x\in Y_N}x^2w_{N}(x)q_i\left(x^2\right)q_j\left(x^2\right)=\frac{\delta_{ij}}{\epsilon_i^2}, 
\end{equation*}
where we choose $\epsilon_i>0$ (the leading coefficient of the $j$th orthonormal polynomial constructed in the Gram-Schmidt procedure).
Once these are at hand, we can mimic the standard random matrix computations \cite[Ch. 5]{Forrester08} and find a determinantal representation of the $m$-point correlation function with kernel
\[
K_{N,k}(x,y)=\sqrt{x^2w_{N}(x)}\sqrt{y^2w_{N}(y)}\sum_{n=0}^{k-1}\epsilon_n^2q_n\left(x^2\right)q_n\left(y^2\right).
\]
The crucial observation pointed out by Forrester is the following alternative construction of the $q_j.$ Consider the monic polynomials $\pi_{2N,j}$ orthogonal with respect to the even weight function $\wq$ on the set of nodes $X_{2N}=Y_N\cup -Y_N,$ 
\begin{equation*}
\sum_{x\in X_{2N}}\wq(x)\pi_{2N,i}\left(x\right)\pi_{2N,j}\left(x\right)=\frac{1}{\gamma_{2N,i}^2}\delta_{ij}. 
\end{equation*}
Since $\wq$ is even, it follows easily from the Gram-Schmidt procedure that $\pi_{2N,2j}(x)$ is even and $\pi_{2N,2j+1}(x)$ is odd. In particular we have $\pi_{2N,2j+1}(x)=xq_j\left(x^2\right)$ for a monic polynomial $q_j$ of degree $j.$ The $q_j$ satisfy
\begin{equation}\label{sympolynom}
\begin{split}
&\sum_{x\in Y_N}x^2w_N(x)q_i\left(x^2\right)q_j\left(x^2\right)\\
=&\sum_{x\in Y_N}w_N(x)\pi_{2N,2i+1}\left(x\right)\pi_{2N,2j+1}\left(x\right)\\
=&\frac{1}{2}\sum_{x\in X_{2N}}\wq(x)\pi_{2N,2i+1}\left(x\right)\pi_{2N,2j+1}\left(x\right)=\frac{1}{2\gamma_{2N,2i+1}^2}\delta_{ij} 
\end{split}
\end{equation}
and hence are the sought for polynomials. For $x,y\in Y_N$ the correlation kernel for the ensemble \eqref{dopesym} can be written as
\begin{equation}\label{kersym}
\begin{split}
K_{N,k}(x,y)=&\sqrt{x^2w_{N}(x)}\sqrt{y^2w_{N}(y)}\sum_{n=0}^{k-1}2\gamma_{2N,2n+1}^2q_n\left(x^2\right)q_n\left(y^2\right)\\
=&\sqrt{\wq(x)}\sqrt{\wq(y)}\sum_{n=0}^{k-1}2\gamma_{2N,2n+1}^2
\pi_{2N,2n+1}(x)\pi_{2N,2n+1}(y)\\
=&2\sqrt{\wq(x)}\sqrt{\wq(y)}\sum_{n=0}^{k-1}
p_{2N,2n+1}(x)p_{2N,2n+1}(y)\\
=&\sqrt{\wq(x)}\sqrt{\wq(y)}\left[ \sum_{n=0}^{2k-1}
p_{2N,n}(x)p_{2N,n}(y)-\sum_{n=0}^{2k-1}
p_{2N,n}(x)p_{2N,n}(-y) \right]\\
=&\Kh(x,y)-\Kh(x,-y),
\end{split}
\end{equation}
the last ``$=$" follows from $\wq(y)=\wq(-y).$ This finishes the proof.
\end{proof}
In the following we will show that for any fixed $0<\theta<b$ and $x,y\in Y_N$ with $x,y>\theta>0$ the summand $\Kh(x,-y)$ tends to zero in the considered limit, and hence the correlation kernels $\Kh$ and $K_{N,k}$ asymptotically coincide. For nodes which are in $O(1/N)$ distance to $0,$ this is in general not the case.

\smallskip
\noindent
\textbf{Remark.} We could also include a node at $0$ in $Y_N$, with $w_N(0)$ some non-negative value.
The symmetrised weight $\widehat{w}_{2N+1}$ is then defined on a set $X_{2N+1}=-Y_N\cup \{0\}\cup Y_N.$ 
Configurations containing $0$ occur with probability $0$ in $DOPE^{\text{sym}}(N,k)$ and the calculations of equation \eqref{sympolynom} also apply in this situation, leading to a kernel $K_{N,k}=\widehat{K}_{2N+1,2k}(x,y)-\widehat{K}_{2N+1,2k}(x,-y).$ This case occurs for example in the half-hexagon model. For notational convenience, we discuss the case of an even number of nodes.\\

%
\section{Analytic assumptions}\label{analyticassumptions}

In order to state our asymptotic results, we review some terminology and results from the monograph \cite{BKMM07}.
In our main application to random tilings, the asymptotic behaviour of $p_{2N,2k}(z)$ and $K_{2N,2k}$ as $N$ and $k$ \emph{simultaneously} tend to infinity plays a crucial role. The results of \cite{BKMM07} are obtained under some technical assumptions \cite[Section 1.2]{BKMM07} on the weight, nodes and number of particles. These are in particular fulfilled in the situation of the (half-) hexagon tilings.
\subsection{Basic assumptions}\label{basicassumptions}
\subsubsection*{The nodes}
The set of nodes $Y_{N}=\{y_{N,0}<\ldots< y_{N,N-1}\}$ is viewed as a subset of a symmetric set $X_{2N}=\{x_{2N,0}<\ldots <x_{2N,2N-1}\}\subset [-b,b]$ of nodes, i.e. $y_{N,n}=x_{2N,N+n},$ $n=0,\ldots,N-1.$ We assume the existence of a \emph{node density function} $\rhoz,$ which is real-analytic in a complex neighbourhood of $[-b,b],$ strictly positive in $[-b,b]$ and satisfies a normalisation condition and a certain quantisation rule, namely
\begin{equation}\label{thenodes}
\int_{-b}^{b}\rhoz(x)\dd x=1\text{ and }\int_{-b}^{x_{2N,n}}\rhoz(x)\dd x=\frac{2n+1}{4N},\; n=0,\ldots,2N-1.
\end{equation}
Due to our symmetry assumption on $X_{2N}$ $\rhoz$ must be even.
\subsubsection*{The weight function}
We assume that we can write the weight function in the form
\begin{equation}\label{theweightfunction}
\wq\left(x_{2N,n}\right)=
(-1)^{2N-n-1}e^{-2NV_{2N}\left(x_{2N,n}\right)} \prod_{
\begin{smallmatrix}
m=0 \\
n\neq m\end{smallmatrix}
}^{2N-1}\left(x_{2N,n}-x_{2N,m}\right)^{-1},
\end{equation}
where $V_{2N}(x)$ is a real-analytic function defined in a neighbourhood $G$ of $[-b,b].$ Furthermore
\begin{equation}\label{potentialfunction}
V_{2N}(x)=V(x)+\frac{\eta(x)}{2N},
\end{equation}
where $V(x)$ is a fixed real-analytic potential function independent of $N$ and
\begin{equation*}
\limsup_{N\rightarrow\infty} \sup_{z\in G} \left|\eta(z)\right|<\infty.
\end{equation*}
As opposed to $V(x),$ the correction $\eta(x)$ may depend on $N.$ Notice that since $\wq$ is even, so is $V_{2N}$ and $V.$ 
\subsubsection*{The number of particles}
The number of particles $k$ and the number of nodes $N$ in $DOPE^{\textnormal{sym}}(N,k)$ are related by
\begin{equation}\label{thedegree}
k=cN+\kappa,
\end{equation}
where $c \in (0,1)$ and $\kappa$ remains bounded as $N\longrightarrow \infty.$

\smallskip
Further assumptions are difficult to express explicitly in terms of the nodes and the weight and are postponed to Section \ref{furtherassumptions}.
\subsection{The associated equilibrium energy problem}\label{equilibrium}
The asymptotic behaviour of $\pi_{N,k}$ is closely related to the asymptotic distribution of zeroes in the interval $[-b,b],$ which in turn can be expressed in terms of quantities arising in a related constrained variational problem \cite{KuiRak99}. Define a real-analytic function $\varphi$ by
\begin{equation}\label{phi}
\varphi(x):=V(x)+\int_{-b}^b\log|x-y|\rhoz(y)\dd y.
\end{equation} 
Note that according to the representation \eqref{theweightfunction} $\wq(x)$ is asymptotically equal to
\begin{equation}
\wq(x)\sim e^{-2N\varphi(x)-\eta(x)}.
\end{equation}

Recall that $\pi_{2N,2k}$ minimises the scalar product \eqref{scalarproduct} among all monic polynomials of degree $2k.$ This relates the asymptotic distribution of zeroes of $\pi_{2N,2k}$ to (the variational derivative of) the quadratic functional of Borel measures on $[a,b]$ defined by
\begin{equation}\label{quadraticfunctional}
E_c[\mu,V]:=\Emu:=c\int_{-b}^b\int_{-b}^b \log\frac{1}{|x-y|}\dd\mu(x)\dd\mu(y)+\int_{-b}^b\varphi(x)\dd\mu(x).
\end{equation}
We are looking for a measure $\mumin$ which minimises $\Emu$ subject to the normalisation condition 
\begin{equation}\label{normalisation}
\int_{-b}^b\dd\mu(x)=1
\end{equation}
and the upper and lower constraints
\begin{equation}\label{upperlowerconstraint}
0\leq \int_{x\in B}\dd\mu(x)\leq \frac{1}{c}\int_{x\in B}\rhoz(x)\dd x \textnormal{ for every a Borel set }B\textnormal{ in }[-b,b].
\end{equation}
We refer to $\mumin$ as the \emph{equilibrium measure}. The constraints are due to the fact that all zeroes of $\pi_{2N,2k}$ are contained in the interval $\left[x_{2N,0},x_{2N,N-1}\right]$ and a closed interval $\left[x_{2N,n},x_{2N,n+1}\right]$ between two consecutive nodes contains at most one zero of $\pi_{2N,2k}.$ Note that minimising $\Emu$ simply subject to the normalisation condition \eqref{normalisation} is formally like seeking a critical point of
\begin{equation*} 
F_c[\mu]=\Emu -l_c\int_{-b}^b\dd\mu(x)
\end{equation*}
with a Lagrange multiplier $l_c:=l_c[V]$. When $\mu=\mumin,$ $l_c$ is a real constant. $\mumin$ is known to be unique and it has a piecewise analytic density $\dmut.$ Points of non-analyticity are finite in number and do not occur in points where the upper and lower constraints 
$\dmut>0$ and $\dmut<1/c$ hold strictly and simultaneously. Under our symmetry assumptions we additionally state for later reference the following easy 
\begin{lemma}\label{lem:muminsymmetric}
 $\mumin$ is invariant under the transformation $x\mapsto -x.$ 
\end{lemma}
\begin{proof}
The symmetry of the weight $\wq$ implies the symmetry of $V_{2N}$ and $V$ (cf. equations \eqref{theweightfunction} and \eqref{potentialfunction}). Hence the external field $\varphi(x)$ defined in \eqref{phi} is even. It follows that the functional $\Emu$ in \eqref{quadraticfunctional} is invariant under the transformation $\mu(x) \mapsto \mu(-x).$ By uniqueness
the symmetry of $\mumin$ follows.  
\end{proof}
\subsection{The equilibrium measure: related quantities and further assumptions: }\label{furtherassumptions}
The constraints give rise to the following
\begin{definition}\label{defbandvoidsaturated}
A \emph{band} is a maximal open subinterval of $[-b,b]$ where $\mumin$ is a measure with a real-analytic density $\dmut$ that satisfies $0<\dmut<\rhoz(x)/c.$ A \emph{void} is a maximal open subinterval of $[-b,b]$ in which $\dmut \equiv 0,$ i.e. meets the \emph{lower} constraint. A \emph{saturated region} is a maximal open subinterval of $[-b,b]$ in which $\dmut \equiv 1/c$ and hence meets the \emph{upper} constraint. If no stress is put on the active constraint, voids and saturated regions are referred to as \emph{gaps}.
\end{definition}
\noindent
As announced in Section \ref{basicassumptions} we make some further assumptions on the weight and the nodes which are expressed in terms of the equilibrium measure, cf. \cite[Section 2.1.2]{BKMM07}. 
Let ${\cal F}$ be the closed set of points $x\in [-b,b],$ where $\dmut=\rhoz(x)/c$ or $\dmut=0$, i.e. one of the constraints \eqref{upperlowerconstraint} is active. We assume that the connected components of ${\cal F}$ have non-empty interior and $-b,b\in {\cal F}$. 
 Furthermore we make the following assumptions on the behaviour of $\dmut$ at an endpoint of a band $I.$ Let $z_0$ be a band end point. If the gap at $z_0$ is a void, then 
\begin{equation*}
\lim_{x\rightarrow z_0,\; x\in I}\frac{1}{ \sqrt{|x-z_0|} }\dmu=K,\quad \text{with}\; 0<K<\infty. 
\end{equation*}
Similarly, if the gap at $z_0$ is a saturated region, we suppose that
\begin{equation*}
\lim_{x\rightarrow z_0,\; x\in I}\frac{1}{ \sqrt{|x-z_0|} }\left(\frac{\rhoz(x)}{c}-\dmu\right)=K,\quad \text{with}\; 0<K<\infty. 
\end{equation*}
So the constraints are met like a square root.

\smallskip
\noindent
\textbf{Remark.} The one-point correlation function in the hexagonal tiling problem is shown to converge pointwise to the density function of the corresponding equilibrium measure in \cite[Theorem 3.12]{BKMM07}, see also Theorems \ref{correlation:band}, \ref{correlation:gap} and \ref{correlation:edge}. In the (half-) hexagon tilings there is a single band corresponding to the intersection of a vertical line with the temperate zone, the surrounding gaps to the intersection of the line with the arctic regions.
%
%

\medskip
\noindent
Finally we define the quantities involved in the asymptotic expressions for $\pi_{N,k}$ in a gap $\Gamma$ and the band $I,$ cf. \cite[Sect. 2.1.4]{BKMM07}. The variational derivative of $\Emu$ evaluated at $\mu=\mumin$ is equal to
\begin{equation}\label{variationalderivative}
\dEmu(x):=\left.\frac{\delta E_{c}[{\mu},V]}{\delta {\mu}}\right|_{{\mu}={\mu}_{\textnormal{min}}^{c}}(x)=-2c\int_{-b}^b\log|x-y|\dd\mumin(x)+\varphi(x).
\end{equation}
We have 
\begin{equation}\label{dEmuminuslc}
\dEmu(x)-l_c \begin{cases}
>0 & \text{if $x$ is in a void,}\\
\equiv 0 &\text{if $x$ is in a band,}\\
<0 &\text{if $x$ is in a saturated region.}
\end{cases}
\end{equation}
The function $\dEmu-l_c$ defined in a gap $\Gamma$ extends analytically from the interior. Furthermore, we denote by $\Lc$ and $\LcI(z)$ the function 
\begin{equation*}
c\int_{-b}^b\log|z-x|\dd\mumin(x),
\end{equation*}
defined for $z,$ in a gap $\Gamma$ or a band $I,$ respectively.
$\Lc$ is analytic in $z$ if $\Re z\in\Gamma$ and $\Im z$ sufficiently small.
$\LcI(z)$ has an analytic continuation to a neighbourhood of $\overline{I}.$ 
\section{Statement of asymptotic results}\label{statements}
We have now introduced all the notation required to state our asymptotic results. The most interesting results are possibly the ``close to zero" assertions in Theorems \ref{correlation:band} and \ref{leftmostparticle} which involve a modified sine kernel and hence state a different asymptotic behaviour in $DOPE^{\text{sym}}.$ The other results are almost verbatim copies of the corresponding results for the $DOPE(2N,2k)$ ensemble in \cite{BKMM07}. Amusingly, the former results can be inferred directly from estimates the of \cite[Chapter 7]{BKMM07} on the correlation kernel while the other results require some additional estimates. The proofs are delivered in the next section. Recall the definition of the $m$-point correlation function of $DOPE^\textnormal{sym}(N,k),$
\[
R_m^{(N,k)}(x_1,\ldots,x_m)=\mathbb{P}(\mbox{particles at each of the sites }x_1,\ldots,x_m).
\]
\smallskip

\noindent
The first results treat the band case and involve the \emph{sine kernel}
\begin{equation}\label{def:sinekernel}
S(\xi,\eta):=\frac{\sin(\pi(\xi-\eta))}{\pi(\xi-\eta)}.
\end{equation}
Extend $S$ to the diagonal by setting $S(\xi,\xi):=1.$ Correlations for nodes close to $0$ are expressed in terms of the kernel given by
\begin{equation*}
S^0(\xi,\eta):=\frac{\sin(\pi(\xi-\eta))}{\pi(\xi-\eta)}-\frac{\sin(\pi(\xi+\eta))}{\pi(\xi+\eta)}.
\end{equation*}
\begin{theorem}\label{correlation:band}
Denote by $R_m^{(N,k)}(x_1,\ldots,x_m)$ the $m$-point correlation function for the ensemble $DOPE^\textnormal{sym}(N,k).$
For $x$ in the interior of a band $I:=(\alpha,\beta)$ define
\begin{equation}\label{delta}
\delta(x):=\left[ c\dmu \right]^{-1}.
\end{equation}
Let $x>0$ and let $\xi_N^{(1)},\ldots,\xi_N^{(m)}$ belong to a fixed bounded set $D\subset \mathbb{R}$ in such  a way that the points defined by
\[
x_j:=x+\xi_N^{(j)}\frac{\delta(x)}{2N},\;j=1,\ldots,m
\]
are all nodes in $Y_N.$ Consequently, $x_j\rightarrow x$ as $N\rightarrow \infty.$ Then there is a constant $C_{D,m}$ depending on $D$ and $m$ such that for sufficiently large $N$ we have
\begin{equation}\label{bandest}
\max_{\xi_N^{(1)},\ldots,\xi_N^{(m)}\in D} \left|
R_m^{(N,k)}(x_1,\ldots,x_m)-\left( \frac{1}{\delta(x)\rhoz(x)}\right)^m
\det\left( S\left(\xi_N^{(i)},\xi_N^{(j)} \right) \right)_{1\leq i,j\leq m}
\right|\leq \frac{C_{D,m}}{N}.
\end{equation}
If $x=0$ and $(-\beta,\beta)$ is a band of $\dmut,$ and $\xi_N^{(1)},\ldots,\xi_N^{(m)}$ are from a fixed bounded set $\tilde{D}\subset \mathbb{R}$ such that
\[
x_j:=\xi_N^{(j)}\frac{\delta(0)}{2N},\;\xi_N^{(j)}>0,    \;j=1,\ldots,m,
\]
are nodes in $Y_N$ then there is a constant $\tilde{C}_{\tilde{D},m}$ depending on $D$ and $m$ such that for sufficiently large $N$
\begin{equation}\label{bandzero}
\max_{\xi_N^{(1)},\ldots,\xi_N^{(m)}\in D} \left|
R_m^{(N,k)}(x_1,\ldots,x_m)-\left(\frac{1}{\rhoz(0)\delta(0)} \right)^m 
\det\left( S^0\left(\xi_N^{(i)},\xi_N^{(j)} \right) \right)_{1\leq i,j\leq m}
\right| 
\leq \frac{\tilde{C}_{\tilde{D},m}}{N}.
\end{equation}

\end{theorem}

\noindent
\textbf{Remark.} The first statement in the preceding theorem is the $DOPE^{\text{sym}}(N,k)$ analogue of \cite[Theorem 3.1]{BKMM07} for $DOPE(2N,2k).$

Another interesting statistic concerning particle systems is the fluctuation of extremal particles. E.g. in random tilings of the half-hexagon this concerns the intersection point of a vertical line with the boundary of the arctic region resp. with the bottom path in the corresponding family of lattice paths. More generally, let $B\subset X_N$ be a set of nodes and $m$ an integer such that $0\leq m \leq \min(\#B,k).$ A well studied statistic is
\begin{equation*}
\begin{split}
A_m^{(N,k)}(B):&=\Prob(\textnormal{there are exactly }m\textnormal{ particles in }B)\\
&=\frac{1}{m!}\left.\left(-\frac{\dd}{\dd t}\right)^m\det\left(\mathbb{I}-t\left.{\cal K}_{N,k}\right|_B \right)\right|_{t=1},
\end{split}
\end{equation*}
where ${\cal K}_{N,k}$ is the operator on $\ell_2(Y_N)$ with kernel $K_{N,k}(x,y)$ and $ \left.{\cal K}_{N,k}\right|_B$ its restriction to $\ell_2(B).$ 
\smallskip

\noindent
In $DOPE^{\text{sym}}(N,k)$ also the behaviour of the leftmost particle is of separate interest. Due to our assumptions in \ref{furtherassumptions}, $0$ lies in the interior of a band or a gap, the gap case being covered by Theorem \ref{correlation:gap}. Hence we assume for the next result that no constraint of $\mumin$ is active in $0$ and denote by ${\cal S}^{0}$ the operator on $\ell_2(\mathbb{Z})$ given by
\begin{equation}\label{def:Ssym}
{\cal S}^{0}_{ij}=\frac{\sin\left(\frac{1}{\delta(0)\rhoz(0)}  \pi(i-j) \right)}{\pi(i-j)} -\frac{\sin\left(\frac{1}{\delta(0)\rhoz(0)} \pi(i+j) \right)}{\pi(i+j)},
\end{equation}  
where $\delta(x)$ is defined in equation \eqref{delta}.
\begin{theorem}\label{leftmostparticle}
Let $x=0$ be an interior point of a band of $\mumin.$ For a fixed set $\mathbb{B}=\{j_1<j_2<\ldots<j_l\}$ of non-negative integers consider the subset $B=\{y_{N,j_1},\ldots,y_{N,j_l} \}\subset Y_N.$ Then
\begin{equation}\label{localoccupation}
A_m^{(N,k)}(B)=\frac{1}{m!}\left.\left(-\frac{\dd}{\dd t}\right)^m\det\left(\mathbb{I}-t\left.{\cal S}^0\right|_{\mathbb{B}} \right)\right|_{t=1}+O(1/N),
\end{equation}
In particular, for the position of the leftmost particle $x_{\textnormal{min}}$ in the  $DOPE^{\textnormal{sym}}(N,k)$ ensemble we have the limiting distribution function
\begin{equation*}
\lim_{N\rightarrow\infty} \Prob\left( x_{\textnormal{min}} \geq \frac{s}{2N\rhoz(0)} \right)
=\det\left(\mathbb{I}- \left.{\cal S}^{0}\right|_{\mathbb{B}(s)}\right),
\end{equation*} 
where $\mathbb{B}(s)$ is the finite set of integers $\left\{0,1,\ldots, \left\lfloor s-1/2 \right\rfloor  \right\}$ and $\left\lfloor s-1/2 \right\rfloor$ denotes the integer part of $s-1/2.$ 

\end{theorem}

\noindent
\textbf{Remark:} The related result for $DOPE(2N,2k)$ is \cite[Theorem 3.2]{BKMM07} on local occupation probabilities in bands.

\smallskip
\begin{theorem}\label{correlation:gap}
Denote by $R_m^{(N,k)}$ the $m$-point correlation function of $DOPE^{\text{sym}}(N,k).$ 
For a fixed closed interval $F$ in a gap $\Gamma$ there are constants $K_F$ and $C_{F,m},$ such that for sufficiently large $N$ and nodes
$x_1,\ldots,x_m\in F\cap Y_N$
\begin{equation*}
\max_{x_1,\ldots,x_m\in F} \left|
R_m^{(N,k)}(x_1,\ldots,x_m)
\right|\leq C_{F,m}\frac{e^{-mK_FN}}{N^m}
\end{equation*}
holds if $\Gamma$ is a void. If $\Gamma$ is a saturated region, we have the estimate 
\begin{equation*}
\max_{x_1,\ldots,x_m\in F} \left|
R_m^{(N,k)}(x_1,\ldots,x_m)-1
\right|\leq C_{F,m}\frac{e^{-K_FN}}{N}.
\end{equation*}
 
%
%

\end{theorem}

\noindent
\textbf{Remark:} The corresponding results for $DOPE(2N,2k)$ are \cite[Theorems 3.3 and 3.5]{BKMM07}. 

\medskip
\noindent
The result for nodes close to a band end point is expressed in terms of the \emph{Airy kernel}
\begin{equation}\label{def:Airykernel}
A(\xi,\eta):=\frac{\textnormal{Ai}(\xi)\textnormal{Ai}^\prime(\eta)-\textnormal{Ai}^\prime(\xi)\textnormal{Ai}(\eta)}{\xi-\eta},
\end{equation}
where $\textnormal{Ai}(x)=\frac{1}{\pi}\int_{0}^\infty \cos (t^3/3+xt) \dd t$ is the Airy function.
\begin{theorem}\label{correlation:edge} Denote by $R_m^{(N,k)}$ the $m$-point correlation function of $DOPE^{\text{sym}}(N,k).$ For each fixed $M>0,$ each integer $m$ and a right band end point $\beta$ separating $I$ from a void, there is a constant $G_{\beta}^m(M)$ such that for sufficiently large $N$ 
\begin{equation*}
\max
\left| R_m^{(N,k)}(x_1,\ldots,x_m)-\left[ \frac{(\pi c B_\beta)^{2/3}}{(2N)^{1/3}\rhoz(\beta)} \right]^m \det\left(A(\xi_N^{(i)},\xi_N^{(j)})  \right)_{1\leq i,j \leq m}    \right| \leq \frac{G_{\beta}^m(M)}{N^{(m+1)/3}},
\end{equation*}
where the $\max$ is taken over nodes $x_1,\ldots,x_m\in X_N$ all satisfying
\[
\beta-M(2N)^{-2/3}<x_j<\beta-M(2N)^{-1/2},
\] 
the constant $B_\beta$ is equal to (cf. Sect. \ref{furtherassumptions})
\[
B_\beta:=\lim_{x\uparrow\beta}\frac{1}{\sqrt{\beta-x}}\dmu>0
\]
and $\xi_N^{(j)}=\left(2N\pi c B_\beta \right)^{2/3}(x_j-\beta).$ If $\beta$ is adjacent to a saturated region, then  
\begin{equation*}
\max
\left| R_m^{(N,k)}(x_1,\ldots,x_m)-1+\frac{(\pi (1-c) \overline{B}_\beta)^{2/3}}{(2N)^{1/3}\rhoz(\beta)}  \sum_{j=1}^mA(\xi_N^{(j)},\xi_N^{(j)})     \right| \leq \frac{H_{\beta}^m(M)}{N^{2/3}},
\end{equation*}
where $H_{\beta}^m(M)$ is a constant, $\xi_N^{(j)}=\left(2N\pi (1-c) \overline{B}_\beta \right)^{2/3}(x_j-\beta)$ and $\overline{B}_\beta$ is equal to 
\[
\overline{B}_\beta=\lim_{x\uparrow\beta}\frac{1}{\sqrt{\beta-x}}
\frac{c}{1-c}\left[\frac{1}{c}\rhoz(x) - \dmu\right]>0.
\]
\end{theorem}

\noindent
\textbf{Remark.} For $DOPE(2N,2k)$ this is \cite[Theorems 3.7 and 3.8]{BKMM07}. Analogous results hold for left band end points.

\smallskip

\noindent

Let $\beta$ be the rightmost band end point and $(\beta,b)$ be a void (resp. saturated region)
Denote by $x_{\textnormal{max}}$ the position of the rightmost particle (resp. hole, i.e. unoccupied node). Since the one-point function of $DOPE(2N,2k)$ converges pointwise to $c\,\dmut$ and $DOPE^{\textnormal{sym}}(N,k)$ has the same asymptotic behaviour in regions bounded away from $0$, one expects $x_{\textnormal{max}}$ near the rightmost band edge $\beta.$ In this domain the correlation kernel is approximated by the Airy kernel. The latter kernel is the correlation kernel of the distribution of eigenvalues of a GUE matrix at the edge of the spectrum and the fluctuations of the largest eigenvalue are governed by the \emph{Tracy-Widom distribution} $\nu$ \cite{TraWid94} whose distribution function is equal to
\[
\nu((\infty,s])=\det\left(\mathbb{I}-\left.{\cal A}\right|_{[s,\infty)}\right),
\]
where $\left.{\cal A}\right|_{[s,\infty)}$ is the trace class operator on $L_2[s,\infty)$ defined by the Airy kernel. The position of the rightmost particle (properly scaled) in $DOPE(2N,2k)$ is proved to be Tracy-Widom-distributed and the proof carries over verbatim to $DOPE^{\textnormal{sym}}(N,k).$ 
\begin{theorem}[Theorem 3.9 in \cite{BKMM07}]\label{Theorem39}
Let the gap adjacent to $(\beta,b)$ be a void. Then the position of the rightmost particle $x_{\textnormal{max}}$ in the ensemble $DOPE^{\textnormal{sym}}(N,k)$ is described by the limiting distribution function
\begin{equation*}
\lim_{N\rightarrow\infty} \Prob\left( x_{\textnormal{max}} \leq \beta+
\frac{s}{(2\pi NcB_\beta)^{2/3}} \right) 
=\det\left(\mathbb{I}-\left.{\cal A}\right|_{[s,\infty)}\right).
\end{equation*} 
If the gap is a saturated region and $x_{\textnormal{max}}$ describes the position of the rightmost unoccupied node (hole), then the same relation holds with $c$ and $B_\beta$ replaced by $1-c$ and $\overline{B}_\beta,$ respectively.
\end{theorem}

%
\section{Proofs of the asymptotic results}\label{proofs}
The proofs of the results of the previous section of course rely heavily on the work on discrete orthogonal polynomials in \cite{BKMM07}. By equation \eqref{kersym} in Section \ref{correlationfunctions} the correlation kernel for $DOPE^{\textnormal{sym}}(N,k)$ is
\[
K_{N,k}(x,y)=\Kh(x,y)-\Kh(x,-y),
\]
where $\Kh$ is the correlation kernel of $DOPE(2N,2k),$ cf. \eqref{Christoffel}. 
The following Lemma from \cite{BKMM07} gives estimates of $\Kh$ in bands.
\begin{fact}[Lemma 7.13 in \cite{BKMM07}]\label{Lemma713}
Fix $x$ in the interior of a band $I=(\alpha,\beta)$ and let
\[
\delta(x):=\left[c \dmu \right]^{-1}.
\]
 Let $\xi_N,\eta_N$ belong to a fixed bounded set $D\subset \mathbb{R}$ in such  a way that the points defined by
\[
w:=x+\xi_N\frac{\delta(x)}{2N},\;z:=x+\eta_N\frac{\delta(x)}{2N}
\]
are both nodes in $X_{2N}.$ Consequently, $w,z\rightarrow x$ as $N\rightarrow \infty.$ Then there is a constant $C_{D}$ depending on $D$ such that for sufficiently large $N$ we have
\begin{equation}\label{Kband}
\max_{\xi_N,\eta_N\in D} 
\left| \Kh(w,z)-\frac{c}{\rhoz(x)} \dmu S\left(\xi_N,\eta_N \right) \right|
\leq \frac{C_{D}}{N}.
\end{equation}
Furthermore, if $x_{2N,i},\;x_{2N,j}\longrightarrow x$ while $i-j$ remains fixed, we have
\begin{equation}
\Kh\left(x_{2N,i},\;x_{2N,j}\right)={\cal S}_{ij}(x)+O(1/N),
\end{equation}
where 
\begin{equation}
{\cal S}_{ij}(x)=  \frac{ \sin \left(\frac{ c  }{ \rhoz(x)  } \dmu \cdot \pi(i-j) \right) }{\pi(i-j)}.
\end{equation}
\qed
\end{fact}
\begin{proof}[Proof of \eqref{bandzero} in Theorem \ref{correlation:band} and Theorem \ref{leftmostparticle}] The previous Lemma yields directly the approximation of $K_{N,k}$ by $S^0,$ and  \eqref{bandzero} follows by taking determinants. For the leftmost particle notice that for $n$ fixed $x_{2N,N+n}\sim \frac{n-1/2}{2N\rhoz(0)}$ as $N\rightarrow\infty.$  The nodes of $X_{2N}$ considered for Theorem \ref{leftmostparticle} hence form the set $\{x_{2N, N+j},\;j=-\lfloor s-1/2 \rfloor-1,\ldots,\lfloor s-1/2 \rfloor \},$ and pairs of nodes converge to zero with the differences of their indices fixed. So we have the approximation of $K_{N,k}$ by ${\cal S}^0$ and by taking determinants Theorem \ref{leftmostparticle} follows.
\end{proof}
\begin{fact}[Lemma 7.16 in \cite{BKMM07}]\label{Lemma716}
Let $\beta$ be a right band end point adjacent to a void. For each fixed $M>0$ there is a constant $C_{\beta}(M)>0,$ such that for $N$ sufficiently large
\begin{equation}\label{Kedge}
\max
\left| \Kh(x,y)-\left[ \frac{(\pi c B_\beta)^{2/3}}{(2N)^{1/3}} \right] A(\xi_N,\eta_N)     \right| \leq \frac{C_{\beta}(M)}{N^{2/3}}
\end{equation}
holds, where the $\max$ is taken over pairs of nodes $x,y\in X_{2N}$ all satisfying
\[
\beta-M(2N)^{-2/3}<x,y<\beta+M(2N)^{-1/2},
\] 
the constant $B_\beta$ is equal to (cf. Sect. \ref{furtherassumptions})
\[
B_\beta:=\lim_{x\uparrow\beta}\frac{1}{\sqrt{\beta-x}}\dmu>0
\]
and $\xi_N=\left(2N\pi c B_\beta \right)^{2/3}(x-\beta)$ and $\eta_N=\left(2N\pi c B_\beta \right)^{2/3}(y-\beta).$\qed
\end{fact}
For the proofs of eq. \eqref{bandest} in Theorem \ref{correlation:band} and Theorem \ref{correlation:edge} we need the following lemma.
\begin{lemma}\label{KeqKhat}
Let $0<\theta<b$ be fixed. If we restrict the correlation kernel $K_{N,k}$ of $DOPE^{\text{sym}}$ to nodes $x,y\in Y_N\cap [\theta,b]$ then Fact \ref{Lemma716} and eq. \eqref{Kband} in Fact \ref{Lemma713} hold with $\Kh$ replaced by $K_{N,k}$ and an adjustment of the respective constants $C_D$ and $C_\beta(M).$
\end{lemma}
%
%
For the proof of the lemma we have to show an $O(N^{-1})$ and an $O(N^{-2/3})$ estimate for the contribution of $\Kh(x,-y)$ to $K_{N,k}(x,y).$ To this end we collect some weaker versions of the Theorems of \cite{BKMM07}, which suffice for our purposes. Recall that by $p_{2N,j}(z),$ $j=0,\ldots,2N-1$ we denote the polynomials orthonormal on $X_{2N}=-Y_N\cup Y_N$ w.r.t. the weight $\wq,$ by $\gamma_{2N,j}$ their leading coefficients and by $\pi_{2N,j}=\gamma_{2N,j}^{-1}p_{2N,j}$ the monic polynomials. According to eq. \eqref{Christoffel}, we only need to care about $\pi_{2N,2k},$ $\pi_{2N,2k-1}$ and $\gamma_{2N,2k-1}.$  
\begin{fact}[Theorem 2.8 in \cite{BKMM07}]\label{Theorem28}
The leading coefficient $\gamma_{2N,2k-1}$ of $p_{2N,2k-1}(z)$ satisfies the asymptotic relation:
\begin{equation*}
\gamma_{2N,2k-1}^2=C_{\gamma_{2N,2k-1} }e^{2Nl_c+\gamma}(1+O(1/N))
\end{equation*}
where $C_{\gamma_{2N,2k-1} }$ is a suitable constant and $\gamma$ remains bounded as $N\longrightarrow \infty.$
\end{fact}
\begin{fact}[Theorems 2.9, 2.13 and 2.15 in \cite{BKMM07}]\label{Theorem213}
In suitable neighbourhoods of the different subintervals of $[-b,b]$ we have the following asymptotic estimates of $\pi_{2N,2k}$ (and $\pi_{2N,2k-1}$). The functions $\Lc$ and $\LcI(z)$ are defined at the end of Section \ref{analyticassumptions}.
\begin{enumerate}
\item
Assume that $J$ is a closed subinterval in a gap $\Gamma.$ Then there is a neighbourhood $K_J$ of $J,$ and a function $A(z)$ analytic on $K_J$ and uniformly bounded in $N,$ such that  
\begin{equation*}
\pi_{2N,2k}(z)=e^{2N\Lc}\left(A(z)+O(1/N)\right)
\end{equation*}
holds.
\item 
Assume that $F$ is a closed subinterval in a band $I.$ Then there is a neighbourhood $K_F$ of $F,$ and a sequence of analytic functions $B_N(z)$ defined on $K_J$ and uniformly bounded in $N,$ such that  
\begin{equation*}
\pi_{2N,2k}(z)=e^{2N\LcI(z)}\left(B_N(z)+O(1/N)\right).
\end{equation*}
\item
Let $z_0$ be a band endpoint. Then there is $r>0$ and sequence of functions $C_N(z)$ analytic in $|z-z_0|<r$ and uniformly bounded for $N\rightarrow\infty,$ such that for $|z-z_0|<r$
\begin{equation*}
\pi_{2N,2k}(z)=e^{2N\LcI(z)}N^{1/6}\left(C_N(z)+O(1/N^{1/3})\right).
\end{equation*}
\end{enumerate}
\qed
\end{fact}
\noindent
The following lemma is inspired by (the proof of) Lemma 7.4 in \cite{BKMM07}.

\begin{lemma}\label{Lemma74}
For any node $x=x_{N,n} \in X_N$ we have the exact formula
\begin{equation*}
\sqrt{\wq(x)}=\frac{1}{\sqrt{2\pi\rhoz(x) N}} e^{-\eta(x)} e^{-N\left(\dEmu(x)-l_c\right)}
e^{-Nl_c} e^{-2Nc \int_a^b \log|x-y|\dd\mumin(y) } T_N(x)^\frac{1}{2}
\end{equation*}
where the function 
\begin{equation*}
T_N(z)=\cos\left( 2\pi N\int_z^b\rhoz(x)\dd x \right)\frac{1}{\prod_{m=0}^{2N-1}(z-x_{2N,m})}e^{2N \int_{-b}^b\log|z-x|\rhoz(x)\dd x}
\end{equation*}
is real analytic in $(-b,b)$ and bounded independently of $N.$
\end{lemma}
\begin{proof}
The exponentials are obtained from the representation \eqref{theweightfunction} of $\wq$ and eq. \eqref{variationalderivative}. The product in \eqref{theweightfunction} is rewritten as
\begin{equation*}
\prod_{
\begin{smallmatrix}
m=0 \\
m\neq n\end{smallmatrix}
}^{2N-1}\left(x_{2N,n}-x_{2N,m}\right)^{-1}=
\lim_{z\rightarrow x_{2N,n}}\frac{z-x_{2N,n}}{\cos\left( 
2\pi N\int_z^b\rhoz(u)\dd u 
\right)}\cdot
\frac{\cos\left( 
2\pi N\int_z^b\rhoz(u)\dd u 
\right)}{\prod_{m=0}^{2N-1}(z-x_{2N,m})}
\end{equation*}
The limit of both fractions exists. De l'H{\^o}pital's rule applied to the first fraction yields
\[
\frac{1}{2\pi\rhoz(x)N\sin\left( 
2\pi N\int_{x}^b\rhoz(u)\dd u 
\right)}=\frac{1}{2\pi\rhoz(x)N(-1)^{2N-n-1}}.
\]
$T_N(z)$ has no poles at the nodes as the zeroes in the denominator are cancelled by the cosine. Furthermore the product in the denominator is asymptotically equal to the exponential for $N\rightarrow\infty.$ This completes the proof.
\end{proof}
%
%
\begin{lemma}[cf. Lemma 7.12 and 7.14 in \cite{BKMM07}]\label{Lemma712}
Denote by $E$ the finite set \linebreak$\{\alpha_0,\beta_0,\ldots,\alpha_h,\beta_h\}$ of band end points.
Let $F$ be a fixed closed subset of the interval $(-b,b)$ such that $F\cap X_{2N}\neq \varnothing$ and $F \cap E=\varnothing$ (and hence all band end points are bounded away from $F$). Then, for $N$ sufficiently large, we have for all $x,y\in F \cap X_{2N}$ the estimate
\begin{equation}\label{eqinLemma712}
\left | (x-y)\Kh(x,y)
e^{N\left(\dEmu(x)-l_c \right)   } e^{N\left(\dEmu(y)-l_c \right)   } \right |<\frac{C}{H(N)}
\end{equation}
with $H(N)=N$ and a constant $C$ only depending on $F.$ Furthermore, if $x,y\in X_{2N}$ are chosen from a sufficiently small neighbourhood $G$ of $E,$ the estimate holds with $H(N)=N^{2/3}$ and the constant $C$ only depending on $G.$  
\end{lemma}
\begin{proof}
Substitute the result of Lemma \ref{Lemma74} and the Facts \ref{Theorem213} and \ref{Theorem28} into the following formula \eqref{Christoffel},
\begin{equation*}
\begin{split}
(x-y)&\Kh(x,y)\\
&=\sqrt{\wq(x)\wq(y)}\cdot\gamma_{2N,2k-1}^2\cdot
\left(\pi_{2N,2k}(x)\pi_{2N,2k-1}(y)-\pi_{2N,2k}(y)\pi_{2N,2k-1}(x)\right)
\end{split}
\end{equation*}
Estimate the uniformly bounded parts by a constant.
\end{proof}

\begin{proof}[Proof of Lemma \ref{KeqKhat}]
First we recall that by Lemma \ref{lem:muminsymmetric} the equilibrium measure is symmetric, hence if $y$ lies in band/void/saturated region then so does $-y.$ The assertions follow directly from Lemma \ref{Lemma712}. Let $x,y\in X_{2N}\cap [\theta,b].$ Then $|x-(-y)|>2\theta$ is bounded away from zero. If $x,y$ are chosen from a closed subset of a band, we know from \eqref{dEmuminuslc} that both $\dEmu(x)-l_c$ and $\dEmu(-y)-l_c$ are zero and hence 
\[
\left |\Kh\left( x,-y \right)\right|<\frac{C}{N}
\]
for a constant $C.$ So after a readjustment of the constant eq. \eqref{Kband} in Lemma \ref{Lemma713} also holds for $K_{N,k}.$ As for the edge estimate, let $M>0$ and $\beta$ be a right band end point separating a band from a void. We consider any pair of nodes $x,y$ with 
\[
\beta-M(2N)^{-2/3}<x,y<\beta+M(2N)^{-1/2}.
\]
Now $\beta$ is strictly positive and for $N$ large $|x-(-y)|$ is close to $2\beta.$ Since the adjacent gap is a void, by \eqref{dEmuminuslc} both $\dEmu(x)-l_c$ and $\dEmu(-y)-l_c$ are $\ge 0.$ Therefore, by Lemma \ref{Lemma712} we have for some constant $D_M$ only depending on $M$
\begin{equation}\label{edgeestimate}
\left |\Kh\left( x,-y \right)\right|<\frac{D_M}{N^{2/3}}. 
\end{equation}
This finishes the proof of Lemma \ref{KeqKhat}.
\end{proof}
Now eq. \eqref{bandest} in Theorem \ref{correlation:band} and the ``void" case of Theorem \ref{correlation:edge} follow by taking determinants. The void case of Theorem \ref{correlation:gap} follows from Lemma \ref{Lemma712} because at the nodes $x,y$ and $-y$ all taken from voids $\dEmu-l_c$ is stricly positive by \eqref{dEmuminuslc} and hence $\Kh(x,y),$ $\Kh(x,-y)$ and $K_{N,k}(x,y)$ must be exponentially small. For the ``saturated region" cases, one invokes the dual ensemble of $DOPE(2N,2k)$ which is obtained by studying ``hole configurations" rather than particles (this notion of duality differs from the one in \cite{KoeSwa98}). The dual ensemble is of $DOPE(2N,2N-2k)$ type \cite[Sections 3.2]{BKMM07}, and the voids for the one are saturated regions for the other and vice versa. Furthermore, off the diagonal the correlation kernel of the dual is the same up to a possible sign change \cite[Propositions 7.2 and 7.3]{BKMM07}, and hence the estimate \eqref{edgeestimate} also holds at the edge of a saturated region. 
The saturated region case of Theorem \ref{correlation:gap} now follows from respective $DOPE(2N,2k)$ result: By \cite[Theorem 3.8]{BKMM07} it is true for $R^{(N,k)}$ replaced by $\Rh$, and since $K_{N,k}$ is simply a perturbation of $\Kh$ of order $O(N^{-2/3})$ so it is for $R^{(N,k)}_m.$ Similarly for Theorem \ref{correlation:gap} where we refer to \cite[Theorem 3.5]{BKMM07} and the perturbation is exponentially small.\qed

\section{Proofs of the half-hexagon statements}\label{halfhexproofs}

\subsection{The Hahn and Associated Hahn ensemble}\label{hahnandassociated}
Let $[-b,b]=[-1/2,1/2]$ and $X_N=\{x_{N,n}=-1/2+(2n+1)/2N,\;n=0,\ldots,N-1 \}.$ The discrete orthogonal polynomial ensemble with weight function
\begin{equation*}
w_N^{\textnormal{AHE}}\left(x_{N,n};P,Q \right)=\frac{1}{n!(P-1+n)!(N-1+n)!(Q-1+N-1-n)!}
\end{equation*}
is called \emph{Associated Hahn Ensemble} (AHE) \cite{BKMM07,Johansson02} with parameters $P$ and $Q.$  
In \cite[Sect. 2.4.2]{BKMM07} the equilibrium measure $\mumin$ for the family $w_N^{\textnormal{AHE}}(\,\cdot\,;AN+1,BN+1),$ $A,B>0,$ with a fixed ratio of the numbers of particles and nodes $c\in(0,1)$ is computed (actually for its particle-hole dual). It turns out that there is exactly one band interval. For $A=B$ it is an interval $(-\beta,\beta)$ enclosed by two gaps of the same type. If $c<c_A=\sqrt{A^2+A}-A$ those two gaps are saturated regions and voids if $c>c_A.$
For the right band end point one has
\begin{equation}\label{beta}
\tilde{\beta}=\frac{\sqrt{c(1-c)(2A+1-c)(2A+2-c)}}{2(A+1-c)}.
\end{equation}
The case $c=c_A$ is exceptional, leading to gaps consisting of single points.
\subsection{The parameters in the $(2k,R,R)$-hexagon}
With the change of variables $z=n-(m/2+k-1/2)$ and the substitutions $2k+m=N$ and $R-m+1=P$ in the weight function $w(z)$ in \eqref{weighthexagon}, we see that $\tilde{w}$ is from the AHE family with the parameters $P$ and $Q$ both equal to $R-m+1.$ 
With view on the general results of the previous sections we consider the parameter $m$ in the distributions in Prop. \ref{hexagondistributions} to be even, $m=2p,$ since this implies that the set $L_{2p}$ has an even number $2N=2k+2p$ of elements. 
For $m=2p+1,$ we have $0\in L_m$ and the remarks at the end of Section \ref{correlationfunctions} allow us to treat that case in the same fashion as below.

\smallskip
Recall that we choose $R$ such that $R/k\rightarrow \lambda$ as $k \rightarrow \infty.$ We scale the $(2k,R,R)$-hexagon by $1/k,$ such that in the limit its intersection with he $x$-axis is the interval $[-\sqrt{3}\lambda/2,\sqrt{3}\lambda/2].$ We want to study the intersection behaviour of the rescaled family of $2k$ lattice paths with a line $x=\tau \in (-\sqrt{3}\lambda/2,0),$ i.e. the parameter $m=2p$ shall satisfy the asymptotic relation $(-R+2p)\sqrt{3}/2k\rightarrow \tau$ for $k\rightarrow \infty$ and hence
\[
2p/k\rightarrow 2\tau/\sqrt{3} +\lambda.
\]
This implies for the ratio $c$ of the numbers of paths (``particles") $2k$ and nodes $2N$ 
\begin{equation}\label{chalfhex}
2k/2N=2k/(2k+2p) \rightarrow c:=\left(\frac{\tau}{\sqrt{3}}+1+\frac{\lambda}{2} \right)^{-1}.
\end{equation}
Moreover for the parameters $P=R-2p+1$ and $A$ we have
\[
(P-1)/2k=2NA/2k \rightarrow \lambda+\tau/\sqrt{3}
\]
and hence we can choose
\[
A=(-\tau/\sqrt{3})c.
\]
The set of nodes $X_{2N}=\{l/k,\;l\in L_{2p}\}$ lies for large $k$ in the interior of the interval $[-1/c,1/c].$ We can compute the band endpoints according to \eqref{beta}. 
After a proper rescaling by $2/c$ the band endpoints read
\begin{equation}\label{ellipse}
\pm\beta=\pm \frac{2}{c}\,\tilde{\beta}=\pm\sqrt{\lambda+1}\sqrt{1-\left(\frac{2}{\sqrt{3}\lambda}\tau\right)^2},\;\tau \in \left[ -\frac{\sqrt{3}\lambda}{2},\frac{\sqrt{3}\lambda}{2}   \right].
\end{equation}
Finally, the function $H(\tau,x)$ in Theorem \ref{theo:hhbulkandrightedge} is obtained by combining Theorem 2.17 and Proposition 2.6 of \cite{BKMM07} and a rescaling by $2/c.$ 

\section{Conclusion}
We have extended the results of the monograph \cite{BKMM07} to a family of discrete probability measures. A special case of the ensembles in question, namely random tilings of the half-hexagon, was introduced in the study of minors of certain random matrices \cite{ForNor09}. Our results on the correlations of tiles on a vertical line and the ``arctic-half-ellipse" result complement the work on this model. 

The continuous counterpart of our general ensemble appears in the study of non-intersecting Brownian excursions \cite{TraWid07} which matches the lattice paths formulation of the half-hexagon model.
In \cite{TraWid07} also the distribution of area below the lowest path is addressed and the first moment thereof is computed. From a combinatorial point of view it would be interesting to (asymptotically) compute the expected area below the lowest path in a family of lattice paths. Can the discrete model shed some light on this distribution as in the case of a single excursion \cite{Takacs91}? 
\section*{Acknowledgments}
The author is indebted to Peter Forrester for sharing his insight leading to the result of Section \ref{correlationfunctions}. The research presented was conducted during the author's appointment at the University of Bielefeld. He gratefully acknowledges financial support by the German Research Council (DFG) within CRC 701. 


\begin{thebibliography}{99}
\bibitem{BKMM07}
J. Baik, T. Kriecherbauer, K.T.-R. McLaughlin and P.D. Miller, \emph{Discrete orthogonal polynomials. Asymptotics and applications,}   Princeton, NJ: Princeton University Press (2007).

\bibitem{Bressoud99}
D.M. Bressoud, \emph{Proofs and Confirmations}, Cambridge: Cambridge University Press (1999).

\bibitem{CohLarPro98}
H. Cohn, M. Larsen and J. Propp, The shape of a typical boxed plane partition, \emph{New York J. Math. } \textbf{4} (1998), 137-165.

\bibitem{ForNor09}
P.J. Forrester and E. Nordenstam, The Anti-Symmetric GUE Minor Process, \emph{Mosc. Math. J.} \textbf{9} (2009), 749-774.

\bibitem{Forrester08}
P.J. Forrester, \emph{Log-Gases and Random Matrices}, Princeton University Press (2010).


\bibitem{Johansson02}
K. Johansson, Non-intersecting paths, random tilings and random matrices, \emph{Probab. Theory Related Fields} \textbf{123} (2002), 225-280.


\bibitem{KoeSwa98}
R. Koekoek and R.F. Swarttouw, The Askey-scheme of hypergeometric orthogonal polynomials and its q-analogue, available online at {\tt{http://fa.its.tudelft.nl/{\textasciitilde}koekoek/askey/}}

\bibitem{KraGutVie00}
C. Krattenthaler, A.J. Guttmann and X.G. Viennot, Vicious walkers, friendly walkers and Young tableaux. II: With a wall, \emph{J. Phys. A} \textbf{33} (2000), No. 48, 8835-8866.

\bibitem{KuiRak99}
A.B.J. Kuijlaars and E.A. Rakhmanov, Zero distributions for discrete orthogonal polynomials, \emph{J. Comput. Appl. Math.} \textbf{99} (1999), No. 2, 255-274.


\bibitem{Mehta04}
M.L. Mehta, \emph{Random Matrices. 3rd ed.}, Amsterdam: Elsevier (2004).


\bibitem{Takacs91}
L. Tak\'acs, A Bernoulli excursion and its various applications, \emph{Adv. Appl. Probab.} \textbf{23} (1991), No. 3, 557-585.


\bibitem{TraWid94}
C.A. Tracy and H. Widom, Level-spacing distributions and the Airy kernel, \emph{Comm. Math. Phys.} \textbf{159} (1994), 151-174.

\bibitem{TraWid07}
C.A. Tracy and H. Widom, Nonintersecting Brownian excursions, \emph{Ann. Appl. Probab.} \textbf{17} (2007), No. 3, 953-979.



\end{thebibliography}

\end{document}